\documentclass[10pt,reqno]{amsart}

\usepackage{amsmath,amssymb,amsthm}
\usepackage[margin=1in]{geometry}
\usepackage[colorlinks=true,linkcolor=blue,citecolor=blue,urlcolor=blue]{hyperref}

\newtheorem{theorem}{Theorem}[section]
\newtheorem{lemma}[theorem]{Lemma}
\newtheorem{proposition}[theorem]{Proposition}
\newtheorem{corollary}[theorem]{Corollary}
\theoremstyle{definition}
\newtheorem*{example}{Example}
\theoremstyle{remark}
\newtheorem{remark}[theorem]{Remark}

\newcommand{\N}{\mathbb{N}}
\newcommand{\C}{\mathbb{C}}
\newcommand{\rmex}{s\text{-}\mathrm{mex}}
\newcommand{\mex}{\mathrm{mex}}
\newcommand{\smex}{\sigma\mathrm{mex}}

\newcommand{\srmex}{\sigma_s\mathrm{mex}}
\newcommand{\sromex}{\sigma_{s,o}\mathrm{mex}}
\newcommand{\sremex}{\sigma_{s,e}\mathrm{mex}}
\newcommand{\Pset}{\mathcal{P}}

\begin{document}

\title[Moments of least $s$-gaps and a conjecture of Baruah and Talukdar]{On the moments of least $s$-gaps of partitions and a conjecture of Baruah and Talukdar}

\author{S. Bhowmick}
\address{Sourav Bhowmick, Department of Mathematics, National Institute of Technology, Raipur, Chhattisgarh 492010.}
\email{souravbhowmick578@gmail.com, sbhowmick.phd2025.maths@nitrr.ac.in}

\author{N. K. Meher}
\address{Nabin Kumar Meher, Department of Mathematics, National Institute of Technology, Raipur, Chhattisgarh 492010.}
\email{mehernabin@gmail.com, nkmeher.maths@nitrr.ac.in}


\begin{abstract}
 Andrews and Newman \cite{AN19, AN20} defined the minimal excludant or mex of a partition. The mex of a partition is the smallest positive integer which is absent from the partition. Baruah, Bhoria, Eyyunni and Maji \cite{BBEM} studied the sum of mex split according to parity, together with its $\ell$-th moments. Ballantine and Merca \cite{BM20} generalized mex to the least $s$-gap, which is the smallest natural number that does not occur minimum $s$ times in the partition. Baruah and Talukdar \cite{BT} conjectured a corresponding asymptotic equivalence between the sums of odd and even least $s$-gaps for every $s>1$, generalizing a theorem of Barman and Singh \cite{BS} for the classical mex. In this article, we derive exact formulas for the $\ell$-th moments of $\rmex(\pi)$ for every fixed $\ell\ge 1$, in terms of partition functions. We give a complete proof of the conjecture of Baruah and Talukdar \cite{BT} for every natural number $s>1$. We also generalize an identity of Hopkins, Sellers and Stanton \cite{HSS} to the least $s$-gap setting.
\end{abstract}

\maketitle

%====================================================================
\section{Introduction}
%====================================================================
Let $n$ be a non-negative integer, let $p(n)$ be the number of partitions of $n$ and $\Pset(n)$ be the collections of all partitions of $n$, with the usual convention $p(0)=1$ and $p(n)=0$ for $n<0$. For a positive integer $\ell$ and complex number $q$ with $|q|<1$, we write
\[
f_\ell := (q^\ell;q^\ell)_\infty = \prod_{i\ge 1}(1-q^{\ell i}).
\]

For a partition $\pi\in\Pset(n)$, the minimal excludant $\mex(\pi)$ is the smallest positive integer missing from $\pi$. This statistic was revived by Andrews and Newman \cite{AN19, AN20}, who studied the total sum
\[
\smex(n) := \sum_{\pi\in\Pset(n)} \mex(\pi)
\]
and proved the identity
\begin{equation}\label{eq:AN}
	\sum_{n\ge 0}\smex(n)q^n = (-q;q)_\infty^2 = \sum_{n\ge 0} D_2(n)q^n,
\end{equation}
where $D_2(n)$ is the number of partitions of $n$ into distinct parts with two colors. The same identity, under the name ``smallest gap'', had earlier been found by Grabner and Knopfmacher \cite{GK}, and was later given a purely combinatorial proof by Ballantine and Merca \cite{BM21}.

Asymptotic formulas for partition functions and related $q$-products have attracted considerable attention in the literature. A fundamental result in this direction is due to Hardy and Ramanujan \cite{HR}, who used the circle method to determine the asymptotic growth of $p(n)$. More precisely,
\begin{equation}\label{eq:HR}
	p(n)\sim \frac{1}{4\sqrt{3}\,n}\exp\left(\pi\sqrt{\frac{2n}{3}}\right), \qquad n\to\infty.
\end{equation}
Grabner and Knopfmacher \cite{GK} also proved the Hardy--Ramanujan-type asymptotic formula
\begin{equation}\label{eq:GK}
	\smex(n)\sim \frac{1}{4}\left(6n^3\right)^{-1/4}\exp\left(\pi\sqrt{\frac{2n}{3}}\right), \qquad n\to\infty.
\end{equation}

Splitting according to the parity of the excludant, Baruah et al.\ \cite{BBEM} introduced these two functions:
\begin{align}
	\sigma_o\mathrm{mex}(n) &:= \sum_{\substack{\pi\in\Pset(n)\\ 2\nmid \mex(\pi)}} \mex(\pi), \label{eq:somex}\\
	\sigma_e\mathrm{mex}(n) &:= \sum_{\substack{\pi\in\Pset(n)\\ 2\mid \mex(\pi)}} \mex(\pi), \label{eq:semex}
\end{align}
and proved the refinement
\begin{align}
	\sum_{n\ge 0}\sigma_o\mathrm{mex}(n)q^n &= \frac{(-q;q)_\infty^2 + (q;q)_\infty^2}{2}, \label{eq:gf-omex}\\
	\sum_{n\ge 0}\sigma_e\mathrm{mex}(n)q^n &= \frac{(-q;q)_\infty^2 - (q;q)_\infty^2}{2}. \label{eq:gf-emex}
\end{align}
In their concluding remarks, they raised the questions to find the Hardy--Ramanujan-type asymptotic formulae for $\sigma_o\mathrm{mex}(n)$ and $\sigma_e\mathrm{mex}(n)$ and for their moments.

In a subsequent work, Barman and Singh \cite{BS} derived the following asymptotic behavior, of Hardy--Ramanujan type, for the arithmetic functions $\sigma_o\mathrm{mex}(n)$ and $\sigma_e\mathrm{mex}(n)$ as $n$ grows large:
\begin{equation}\label{eq:BS}
	\sigma_o\mathrm{mex}(n)\sim\sigma_e\mathrm{mex}(n)\sim \frac{1}{8\sqrt[4]{6n^3}}\exp\left(\pi\sqrt{\frac{2n}{3}}\right), \qquad \text{as } n\to\infty.
\end{equation}

Ballantine and Merca \cite{BM20} generalized $\mex(\pi)$ to $\rmex(\pi)$, which is the least positive integer that does not occur at least $s$ times as a part of $\pi$ (so $1\text{-}\mex(\pi)=\mex(\pi)$), where $s\in \mathbb{N}.$

\begin{example}
	Take $\pi=(3,2,2,1,1,1)$, a partition of $n=10$; then $1\text{-}\mex(\pi)=4$, $2\text{-}\mex(\pi)=3$, $3\text{-}\mex(\pi)=2$, and $\rmex(\pi)=1$ for $s\ge 4$.
\end{example}

Writing $\srmex(n):=\sum_{\pi\in\Pset(n)}\rmex(\pi)$, Ballantine and Merca \cite{BM20} found the generating function
\begin{equation}\label{eq:gf-rmex}
	\sum_{n\ge 0}\srmex(n)q^n = \frac{f_{2s}^2}{f_1 f_s},
\end{equation}
which was re-derived by Baruah and Talukdar \cite{BT}. They also obtained the generating functions for $\sromex(n)$ and $\sremex(n)$, which are least $s$-gap analogues of $\sigma_o\mathrm{mex}(n)$ and $\sigma_e\mathrm{mex}(n)$ respectively:
\begin{align}
	\sum_{n=0}^\infty \sromex(n)q^n &= \frac{1}{2}\big(\mathcal{A}_s(q) + \mathcal{B}_s(q)\big), \label{eq:gf-romex}\\
	\sum_{n=0}^\infty \sremex(n)q^n &= \frac{1}{2}\big(\mathcal{A}_s(q)-\mathcal{B}_s(q)\big). \label{eq:gf-remex}
\end{align}
where
\begin{equation}\label{eq:Ar}
\mathcal{A}_s(q) = \frac{f_{2s}^2}{f_1f_s},
\end{equation}
and
\begin{equation}\label{eq:Br}
\mathcal{B}_s(q) = \frac{f_s^3}{f_1}.
\end{equation}
Baruah and Talukdar \cite{BT} also proved the asymptotic formula for $\srmex(n)$
\begin{equation}\label{eq:BT-asym}
	\srmex(n)\sim \frac{1}{4}\left(6n^3s^2\right)^{-1/4}\exp\left(\pi\sqrt{2n/3}\right) \qquad \text{as } n\to\infty.
\end{equation}

Baruah and Talukdar \cite{BT} conjectured that \eqref{eq:BS} extends to every fixed $s>1$. Let us define
\begin{align}
	\sigma^{(\ell)}_s\mathrm{mex}(n) &:= \sum_{\pi\in\Pset(n)}\big(\rmex(\pi)\big)^\ell, \label{eq:mom}\\
	\overline{\sigma}^{(\ell)}_s\mathrm{mex}(n) &:= \sum_{\pi\in\Pset(n)}(-1)^{\rmex(\pi)-1}\big(\rmex(\pi)\big)^\ell, \label{eq:mom-alt}\\
	\sigma^{(\ell)}_{s,o}\mathrm{mex}(n) &:= \sum_{\substack{\pi\in\Pset(n)\\ 2\nmid\rmex(\pi)}}\big(\rmex(\pi)\big)^\ell, \label{eq:mom-o}\\
	\sigma^{(\ell)}_{s,e}\mathrm{mex}(n) &:= \sum_{\substack{\pi\in\Pset(n)\\ 2\mid\rmex(\pi)}}\big(\rmex(\pi)\big)^\ell. \label{eq:mom-e}
\end{align}
For $\ell=1$ these reduce to $\srmex(n)$, $\sromex(n)$, $\sremex(n)$. At $s=1$ they reduce to the functions of Andrews--Newman \cite{AN19, AN20} and Baruah et al.\ \cite{BBEM} above.

Let $p_{\rmex}(m,n)$ denote the number of partitions of $n$ with $\rmex(\pi)=m$, and let 
\[
t_s(m) := s\cdot\frac{m(m+1)}{2}, \qquad m\ge 0.
\]

We now state all the results of this paper. Their proofs are given in Section~\ref{sec:proofs}, and all lemmas, proposition used in the proofs are collected in Section~\ref{sec:prelim}.

Our first result is the least-$s$-gap generalization of an equation for $p_{\rmex}(m,n)$ that goes back to Baruah et al.\ \cite{BBEM}, at $s=1$.

\begin{theorem}\label{thm:pmex}
For fixed $n\ge 0$ and $z\in\C$,
\[
\sum_{m\ge 1}p_{\rmex}(m,n)z^m = p(n) + (z-1)\sum_{m\ge 0}p\big(n-t_s(m)\big)z^m.
\]
\end{theorem}

\begin{theorem}\label{thm:pmex-k}
For $\ell,n\in\N$, $s\ge 1$ and $z\in\C$,
\[
\sum_{m\ge 1}p_{\rmex}(m,n)\,m^\ell z^m = \sum_{m\ge 0}\Big[(m+1)^\ell z^{m+1} - m^\ell z^m\Big]\,p\big(n-t_s(m)\big).
\]
\end{theorem}

\begin{theorem}\label{thm:moments}
For $\ell,n\in\N$ and $s\ge 1$,
\begin{align}
\sigma^{(\ell)}_s\mathrm{mex}(n) &= \sum_{m\ge 0}\Big((m+1)^\ell - m^\ell\Big)\,p\big(n-t_s(m)\big), \label{eq:mom-formula}\\
\overline{\sigma}^{(\ell)}_s\mathrm{mex}(n) &= \sum_{m\ge 0}(-1)^m\Big((m+1)^\ell + m^\ell\Big)\,p\big(n-t_s(m)\big). \label{eq:mom-alt-formula}
\end{align}
\end{theorem}

\begin{corollary}\label{cor:k1}
For $s\ge 1$ and $n\ge 1$,
\begin{align*}
\srmex(n) &= \sum_{m\ge 0}p\big(n-t_s(m)\big),\\
\overline{\sigma}_s\mathrm{mex}(n) &= \sum_{m\ge 0}(-1)^m(2m+1)\,p\big(n-t_s(m)\big).
\end{align*}
\end{corollary}

\begin{theorem}\label{thm:oe-moments}
For positive integers $\ell,n$ and $s\ge 1$, we have
\begin{align}
\sigma^{(\ell)}_{s,o}\mathrm{mex}(n) &= \sum_{m\ge 0}\Big[\epsilon_m(m+1)^\ell - (1-\epsilon_m)m^\ell\Big]\,p\big(n-t_s(m)\big), \label{eq:o-mom}\\
\sigma^{(\ell)}_{s,e}\mathrm{mex}(n) &= \sum_{m\ge 0}\Big[(1-\epsilon_m)(m+1)^\ell - \epsilon_m m^\ell\Big]\,p\big(n-t_s(m)\big), \label{eq:e-mom}
\end{align}
where $\epsilon_m = \frac{1+(-1)^m}{2}$.
\end{theorem}

\begin{corollary}\label{cor:oe}
For $s\ge 1$ and $n\ge 0$,
\begin{align*}
\sromex(n) &= \sum_{m\ge 0}(2m+1)\Big[p\big(n-t_s(2m)\big) - p\big(n-t_s(2m+1)\big)\Big],\\
\sremex(n) &= \sum_{m\ge 0}(2m+2)\Big[p\big(n-t_s(2m+1)\big) - p\big(n-t_s(2m+2)\big)\Big].
\end{align*}
\end{corollary}

We can now state our resolution of Conjecture~29 of Baruah and Talukdar \cite{BT}.

\begin{theorem}\label{thm:main}
For every fixed positive integer $s>1$, as $n\to\infty$, we have
\[
\sromex(n)\sim\sremex(n)\sim\frac{1}{8}\left(6n^3s^2\right)^{-1/4}\exp\left(\pi\sqrt{\frac{2n}{3}}\right).
\]
\end{theorem}
Our next theorem generalizes an identity of Hopkins, Sellers and Stanton \cite{HSS} to the least $s$-gap setting.
\begin{theorem}\label{thm:HSS}
	Let $o^{(s)}_1(n)$ (resp.\ $o^{(s)}_3(n)$) be the number of $\pi\in\Pset(n)$ with $\rmex(\pi)\equiv 1$ (resp.\ $\equiv 3$) modulo $4$. For every $s\ge 1$,
\[
\sum_{n\ge 0}\Big(o^{(s)}_1(n) - o^{(s)}_3(n)\Big)q^n = \frac{f_sf_{4s}}{f_1f_{2s}}.
\]
\end{theorem}

\begin{remark}
At $s=1$, Theorem~\ref{thm:HSS} reduces to
\[
\sum_{n\ge 0}\big[o_1(n)-o_3(n)\big]q^n = \frac{f_4}{f_2} = (-q^2;q^2)_\infty = \sum_{n\ge 0} q(n)q^{2n},
\]
where $q(n)$ counts partitions of $n$ into distinct parts. This is Proposition~2.8 of Baruah et al.\ \cite{BBEM}, and also proved analytically by Hopkins, Sellers and Yee \cite{HSY}.
\end{remark}

%====================================================================
\section{Preliminaries}\label{sec:prelim}
%====================================================================

In this section we collect all the lemmas, together with their proofs, that are used in Section~\ref{sec:proofs}.

\begin{lemma}\label{lem:JTP}
For $|ab|<1$,
\[
f(a,b) := \sum_{n=-\infty}^{\infty}a^{n(n+1)/2}b^{n(n-1)/2} = (-a;ab)_\infty(-b;ab)_\infty(ab;ab)_\infty.
\]
\end{lemma}

\begin{proof}
This is the classical Jacobi triple product identity. See Berndt \cite{Berndt}.
\end{proof}

\begin{lemma}\label{lem:eta}
For any fixed positive integer $\ell$, as $y\to 0^+$,
\begin{equation}\label{eq:fk-asym}
f_\ell(e^{-y}) = (e^{-\ell y};e^{-\ell y})_\infty \sim \sqrt{\frac{2\pi}{\ell y}}\exp\left(-\frac{\pi^2}{6\ell y}\right).
\end{equation}
\end{lemma}

\begin{proof}
 See Koblitz \cite{Koblitz}.
\end{proof}

\begin{theorem}[Ingham's Tauberian theorem]\label{thm:Ingham}
Let $G(q)=\sum_{n\ge 0}g(n)q^n$ have radius of convergence $1$, where $\{g(n)\}$ is a weakly increasing sequence of non-negative reals. Suppose there are constants $\mu,\nu\in\mathbb{R}$ and $\lambda>0$ with
\[
G(e^{-y})\sim\mu\, y^{\nu}\exp\left(\frac{\lambda}{y}\right), \qquad y\to 0^+.
\]
Then
\[
g(n)\sim\frac{\mu}{2\sqrt{\pi}}\,\lambda^{\frac{2\nu+1}{4}}\,n^{-\frac{2\nu+3}{4}}\exp\left(2\sqrt{\lambda n}\right), \qquad \text{as } n\to\infty.
\]
\end{theorem}

\begin{proof}
See Theorem~1 of Ingham \cite{Ingham}.
\end{proof}

\begin{proposition}\label{lem:monotone}
For every fixed $s\ge 1$, the sequences $\{\sromex(n)\}_{n\ge 1}$ and $\{\sremex(n)\}_{n\ge 1}$ are weakly increasing.
\end{proposition}

\begin{proof}
	Let $\mathcal{P}_{s,o}(n)$ (resp.\ $\mathcal{P}_{s,e}(n)$) be the collections of partitions $\pi$ of $n$ with $s\text{-}\mathrm{mex}(\pi)$ odd (resp.\ even).
	
	\noindent\textbf{The odd case.} Fix $n\ge 1$ and define $\Phi:\mathcal{P}_{s,o}(n)\to\mathcal{P}_{s,o}(n+1)$ as follows: for $\pi\in\mathcal{P}_{s,o}(n)$
	\begin{enumerate}
		\item[(1)] if $s\text{-}\mathrm{mex}(\pi)\neq 1$, let $\Phi(\pi)$ be $\pi$ with a new part $1$ adjoined.
		\item[(2)] if $s\text{-}\mathrm{mex}(\pi)=1$, let $\Phi(\pi)$ be $\pi$ with its largest part increased by $1$.
	\end{enumerate}
	
	\begin{itemize}
		\item \emph{$\Phi$ preserves $s$-mex}: If $s\text{-}\mathrm{mex}(\pi)\neq 1$, part $1$ already occurs at least $s$ times in $\pi$. Therefore adjoining one more copy of $1$ does not change whether any part occurs $\ge s$ times, so $s\text{-}\mathrm{mex}(\Phi(\pi))=s\text{-}\mathrm{mex}(\pi)$.\\
		If $s\text{-}\mathrm{mex}(\pi)=1$, part $1$ occurs fewer than $s$ times in $\pi$. Let $m$ be the largest part of $\pi$. Increasing $m$ to $m+1$ leaves the multiplicity of part $1$ unchanged if $m\ge 2$, and decreases it by one (still $<s$) if $m=1$. Either way part $1$ still occurs fewer than $s$ times in $\Phi(\pi)$. So $s\text{-}\mathrm{mex}(\Phi(\pi))=1=s\text{-}\mathrm{mex}(\pi)$.
		\item \emph{$\Phi$ is injective}: 
	 Let $\pi_1,\pi_2\in\mathcal{P}_{s,o}(n)$ with $\Phi(\pi_1)=\Phi(\pi_2)=\lambda$. 
	 
	\emph{Case (i): $s\text{-}\mathrm{mex}(\lambda)\neq 1$.} Then
	$	\pi_1\cup\{1\}=\lambda=\pi_2\cup\{1\}
	\;\Longrightarrow\; \pi_1=\lambda\setminus\{1\}=\pi_2.$
	
	\emph{Case (ii): $s\text{-}\mathrm{mex}(\lambda)=1$.} Let $L_i$ be one of the largest part of $\pi_i$ for $i=1,2$. Then
	\[
	\big(\pi_1\setminus\{L_1\}\big)\cup\{L_1+1\}=\lambda=\big(\pi_2\setminus\{L_2\}\big)\cup\{L_2+1\}.
	\]
	Comparing the largest parts of $\lambda$, we get
	$L_1+1=L_2+1 \;\Longrightarrow\; L_1=L_2
	\Longrightarrow \pi_1=\pi_2.$
	 Therefore $\Phi(\pi_1)=\Phi(\pi_2)\Longrightarrow\pi_1=\pi_2$, which shows that $\Phi$ is injective. Hence, for $n\ge 1$
	\[
	\sigma_{s,o}\mathrm{mex}(n)=\sum_{\pi\in\mathcal{P}_{s,o}(n)} s\text{-}\mathrm{mex}(\pi)
	=\sum_{\pi\in\mathcal{P}_{s,o}(n)} s\text{-}\mathrm{mex}(\Phi(\pi))
	\le\sum_{\mu\in\mathcal{P}_{s,o}(n+1)} s\text{-}\mathrm{mex}(\mu)
	=\sigma_{s,o}\mathrm{mex}(n+1).
	\]
\end{itemize}
	\noindent\textbf{The even case.} For $n\ge 1$, define $\Psi:\mathcal{P}_{s,e}(n)\to\mathcal{P}_{s,e}(n+1)$ by letting $\Psi(\pi)$ be $\pi$ with a new part $1$ adjoined.
	\begin{itemize}
	\item \emph{$\Psi$ preserves $s$-mex}: If $s\text{-}\mathrm{mex}(\pi)=j$ is even (so $j\ge 2$), part $1$ already occurs at least $s$ times in $\pi$. So adjoining one more copy of $1$ leaves every part's multiplicity status unchanged. Hence $s\text{-}\mathrm{mex}(\Psi(\pi))=j=s\text{-}\mathrm{mex}(\pi)$.
	\item \emph{$\Psi$ is injective}: Let $\pi_1,\pi_2\in\mathcal{P}_{s,e}(n)$ with $\Psi(\pi_1)=\Psi(\pi_2)=\lambda$. Then
	\[
	\pi_1\cup\{1\}=\lambda=\pi_2\cup\{1\}
	\;\Longrightarrow\; \pi_1=\lambda\setminus\{1\}=\pi_2.
	\]
	Therefore $\Psi(\pi_1)=\Psi(\pi_2)\Longrightarrow\pi_1=\pi_2$, which shows that $\Psi$ is injective. Hence, for $n\ge 1$
	\[
	\sigma_{s,e}\mathrm{mex}(n)=\sum_{\pi\in\mathcal{P}_{s,e}(n)} s\text{-}\mathrm{mex}(\pi)
	=\sum_{\pi\in\mathcal{P}_{s,e}(n)} s\text{-}\mathrm{mex}(\Psi(\pi))
	\le\sum_{\mu\in\mathcal{P}_{s,e}(n+1)} s\text{-}\mathrm{mex}(\mu)
	=\sigma_{s,e}\mathrm{mex}(n+1).
	\]
\end{itemize}
	Therefore, $\{\sigma_{s,o}\mathrm{mex}(n)\}_{n\ge 1}$ and $\{\sigma_{s,e}\mathrm{mex}(n)\}_{n\ge 1}$ are weakly increasing.
\end{proof}
\begin{lemma}\label{lem:Mr}
	As $y\to 0^+$, we have
	\[
	\mathcal{A}_s(e^{-y})=\frac{(e^{-2sy};e^{-2sy})_\infty^2}{(e^{-y};e^{-y})_\infty\,(e^{-sy};e^{-sy})_\infty}
	\sim\frac{1}{2\sqrt{s}}\exp\left(\frac{\pi^2}{6y}\right).
	\]
\end{lemma}

\begin{proof}
	%	From \eqref{eq:Mr}, we have
	%	\[
	%	\mathcal{A}_s(e^{-y})=\frac{(e^{-2sy};e^{-2sy})_\infty^2}{(e^{-y};e^{-y})_\infty\,(e^{-sy};e^{-sy})_\infty}.
	%	\]
	Taking $\ell=2s$, $\ell=1$ and $\ell=s$ in \eqref{eq:fk-asym}, we get
	\[
	(e^{-2sy};e^{-2sy})_\infty^2\sim\frac{\pi}{sy}\exp\left(-\frac{\pi^2}{6sy}\right)
	\]
	and
	\[
	(e^{-y};e^{-y})_\infty\,(e^{-sy};e^{-sy})_\infty\sim\frac{2\pi}{\sqrt{s}\,y}\exp\left(-\frac{\pi^2}{6y}\left(1+\frac{1}{s}\right)\right).
	\]
	Hence,
	\[
	\frac{(e^{-2sy};e^{-2sy})_\infty^2}{(e^{-y};e^{-y})_\infty\,(e^{-sy};e^{-sy})_\infty}
	\sim\frac{1}{2\sqrt{s}}\exp\left(\frac{\pi^2}{6y}\left(1+\frac{1}{s}-\frac{1}{s}\right)\right)
	=\frac{1}{2\sqrt{s}}\exp\left(\frac{\pi^2}{6y}\right). \qedhere
	\]
\end{proof}

\begin{lemma}\label{lem:Nr}
	As $y\to 0^+$, we have
	\[\mathcal{B}_s(e^{-y})=\frac{(e^{-sy};e^{-sy})_\infty^3}{(e^{-y};e^{-y})_\infty}
	\sim\frac{2\pi}{s^{3/2}y}\exp\left(\frac{\pi^2}{y}\left(\frac{1}{6}-\frac{1}{2s}\right)\right).
	\]
	Consequently, for every fixed $s\ge 1$,
	\[
	\lim_{y\to 0^+}\frac{\mathcal{B}_s(e^{-y})}{\mathcal{A}_s(e^{-y})}=0.
	\]
\end{lemma}

\begin{proof}
	%	From \eqref{eq:Nr}, we have
	%	\[
	%	\mathcal{B}_s(e^{-y})=\frac{(e^{-sy};e^{-sy})_\infty^3}{(e^{-y};e^{-y})_\infty}.
	%	\]
	Taking $\ell=s$ and $\ell=1$ in \eqref{eq:fk-asym}, we get
	\[
	(e^{-sy};e^{-sy})_\infty^3\sim\left(\frac{2\pi}{sy}\right)^{3/2}\exp\left(-\frac{\pi^2}{2sy}\right), \qquad
	(e^{-y};e^{-y})_\infty\sim\sqrt{\frac{2\pi}{y}}\exp\left(-\frac{\pi^2}{6y}\right).
	\]
	Hence
	\[
	\mathcal{B}_s(e^{-y})\sim\frac{2\pi}{s^{3/2}y}\exp\left(\frac{\pi^2}{y}\left(\frac{1}{6}-\frac{1}{2s}\right)\right).
	\]
	Now, from Lemma~\ref{lem:Mr}, we have
	\[
	\mathcal{A}_s(e^{-y})\sim\frac{1}{2\sqrt{s}}\exp\left(\frac{\pi^2}{6y}\right).
	\]
	Therefore,
	\[
	\frac{\mathcal{B}_s(e^{-y})}{\mathcal{A}_s(e^{-y})}
	\sim\frac{2\pi}{s^{3/2}y} 2\sqrt{s}
	\exp\left(\frac{\pi^2}{6y}-\frac{\pi^2}{2sy}-\frac{\pi^2}{6y}\right)
	=\frac{4\pi}{sy}\exp\left(-\frac{\pi^2}{2sy}\right).
	\]
	Putting $t=1/y$, we see that $t\to\infty$ as $y\to 0^+$. Then
	\[
	\frac{4\pi}{sy}\exp\left(-\frac{\pi^2}{2sy}\right)=\frac{4\pi}{s}\cdot\frac{t}{e^{\frac{\pi^2}{2s}t}}.
	\]
	By L'H\^{o}pital's rule, we get
	\[
	\lim_{t\to\infty}\frac{t}{e^{\frac{\pi^2}{2s}t}}
	=\lim_{t\to\infty}\frac{1}{\frac{\pi^2}{2s}\,e^{\frac{\pi^2}{2s}t}}
	=0.
	\]
	Therefore, we obtain
	\[
	\lim_{y\to 0^+}\frac{\mathcal{B}_s(e^{-y})}{\mathcal{A}_s(e^{-y})}
	=0. \qedhere
	\]
\end{proof}

\begin{lemma}\label{lem:gf-asym}
	As $y\to 0^+$, we have
	\[
	\sum_{n\ge 0}\sromex(n)e^{-ny}\sim\sum_{n\ge 0}\sremex(n)e^{-ny}\sim\frac{1}{4\sqrt{s}}\exp\left(\frac{\pi^2}{6y}\right).
	\]
\end{lemma}

\begin{proof}
	By Lemma \ref{lem:Nr},
	\[
	\lim_{y\to 0^+}\frac{\mathcal{B}_s(e^{-y})}{\mathcal{A}_s(e^{-y})}=0.
	\]
	Therefore
	\[
	\frac{\mathcal{A}_s(e^{-y})\pm\mathcal{B}_s(e^{-y})}{\mathcal{A}_s(e^{-y})}
	=1\pm\frac{\mathcal{B}_s(e^{-y})}{\mathcal{A}_s(e^{-y})}\longrightarrow 1\pm 0=1
	\qquad\text{as } y\to 0^+.
	\]
	Hence
	\[
	\mathcal{A}_s(e^{-y})\pm\mathcal{B}_s(e^{-y})\sim\mathcal{A}_s(e^{-y})\qquad\text{as } y\to 0^+.
	\]
	Therefore from \eqref{eq:gf-romex}, \eqref{eq:gf-remex} and Lemma~\ref{lem:Mr}, we obtain
	\[
	\sum_{n\ge 0}\sigma_{s,o}\mathrm{mex}(n)e^{-ny}\sim\sum_{n\ge 0}\sigma_{s,e}\mathrm{mex}(n)e^{-ny}\sim\frac{1}{4\sqrt{s}}\exp\left(\frac{\pi^2}{6y}\right).
	\]
	
\end{proof}

%====================================================================
\section{Proofs}\label{sec:proofs}
%====================================================================

\begin{proof}[Proof of Theorem~\ref{thm:pmex}]
From \cite{BT}, we have
\[
M(z,q) := \sum_{n\ge 0}\sum_{m\ge 1}p_{\rmex}(m,n)z^mq^n = \frac{1}{f_1}\sum_{m\ge 1}z^mq^{t_s(m-1)}\big(1-q^{sm}\big).
\]
Splitting the sum on the right and re-indexing with $j=m-1$,
\begin{align*}
\sum_{m\ge 1}z^mq^{t_s(m-1)}\big(1-q^{sm}\big)
&= z\sum_{j\ge 0}z^jq^{t_s(j)} - \sum_{m\ge 1}z^mq^{t_s(m)}\\
&= z\sum_{j\ge 0}z^jq^{t_s(j)} - \left(\sum_{m\ge 0}z^mq^{t_s(m)} - 1\right)\\
&= 1 + (z-1)\sum_{m\ge 0}z^mq^{t_s(m)}.
\end{align*}
Hence
\[
\sum_{n\ge 0}\sum_{m\ge 1}p_{\rmex}(m,n)z^mq^n = \sum_{n\ge 0}p(n)q^n + (z-1)\sum_{m\ge 0}z^m\sum_{n\ge 0}p\big(n-t_s(m)\big)q^n.
\]
Now equating the coefficients of $q^n$ on both sides gives the claim.
\end{proof}

\begin{proof}[Proof of Theorem~\ref{thm:pmex-k}]
We will use the differential operator $D_z := z\,\dfrac{\partial}{\partial z}$. Applying it once to $z^m$ gives $D_z(z^m)=m z^m$, and applying it $\ell$ times gives
$D_z^\ell(z^m)=m^\ell z^m$ for $m\ge 0$. Note that, on the right-hand side $p\big(n-t_s(m)\big)=0$ as soon as $t_s(m)>n$, so only finitely many terms survive. On the left-hand side, every partition $\pi$ of $n$ satisfies $\rmex(\pi)\le n+1$, so $p_{\rmex}(m,n)=0$ for $m>n+1$. Since we are dealing with finite sums, we can apply $D_z^\ell$ term by term. Therefore, we obtain
\[
\sum_{m\ge 1}p_{\rmex}(m,n)\,D_z^\ell(z^m) = \sum_{m\ge 0}p\big(n-t_s(m)\big)\,D_z^\ell\big(z^{m+1}-z^m\big).
\]
Finally, since $D_z^\ell(z^m)=m^\ell z^m$ and $D_z^\ell(z^{m+1})=(m+1)^\ell z^{m+1}$, this becomes
\[
\sum_{m\ge 1}p_{\rmex}(m,n)\,m^\ell z^m = \sum_{m\ge 0}\Big[(m+1)^\ell z^{m+1} - m^\ell z^m\Big]\,p\big(n-t_s(m)\big).
\]
This proves our claim.
\end{proof}

\begin{proof}[Proof of Theorem~\ref{thm:moments}] Note that, at $z=1$, $D_z^\ell(z^m)=m^\ell$, so the left-hand side of Theorem~\ref{thm:pmex-k} becomes
\[
\sum_{m\ge 1}p_{\rmex}(m,n)m^\ell = \sigma^{(\ell)}_s\mathrm{mex}(n).
\]
 Therefore, \[
\sigma^{(\ell)}_s\mathrm{mex}(n)=\sum_{m\ge 0}\Big((m+1)^\ell-m^\ell\Big)\,p\big(n-t_s(m)\big).
\]
At $z=-1$, $D_z^\ell(z^m)=m^\ell(-1)^m$. Therefore, the left-hand side of Theorem~\ref{thm:pmex-k} becomes
\[
\sum_{m\ge 1}p_{\rmex}(m,n)m^\ell(-1)^m = -\overline{\sigma}^{(\ell)}_s\mathrm{mex}(n).
\]
On the right-hand side, we have
\[
(m+1)^\ell(-1)^{m+1} - m^\ell(-1)^m = -(-1)^m\Big((m+1)^\ell+m^\ell\Big).
\]
Therefore,
\[
\overline{\sigma}^{(\ell)}_s\mathrm{mex}(n)=\sum_{m\ge 0}(-1)^m\Big((m+1)^\ell+m^\ell\Big)\,p\big(n-t_s(m)\big).
\]
\end{proof}

\begin{proof}[Proof of Corollary~\ref{cor:k1}]
Putting $\ell=1$ in Theorem~\ref{thm:moments}, we get the first identity. For the second, since $(-1)^m\big((m+1)+m\big)=(-1)^m(2m+1)$, we obtain the second identity.
\end{proof}

\begin{proof}[Proof of Theorem~\ref{thm:oe-moments}]
Adding the two identities of Theorem~\ref{thm:moments} gives
\begin{equation}\label{eq:add}
\sigma^{(\ell)}_s\mathrm{mex}(n) + \overline{\sigma}^{(\ell)}_s\mathrm{mex}(n)
= \sum_{m\ge 0}\Big[\big(1+(-1)^m\big)(m+1)^\ell - \big(1-(-1)^m\big)m^\ell\Big]\,p\big(n-t_s(m)\big).
\end{equation}
The left-hand side equals
\[
\sum_{\pi\in\Pset(n)}\Big[\rmex(\pi)^\ell - (-1)^{\rmex(\pi)}\rmex(\pi)^\ell\Big]
= 2\sum_{\substack{\pi\in\Pset(n)\\ \rmex(\pi)\ \mathrm{odd}}}\big(\rmex(\pi)\big)^\ell
= 2\sigma^{(\ell)}_{s,o}\mathrm{mex}(n).
\]
By substituting the above expression in \eqref{eq:add}, we get
\[
\sum_{m\ge 0}\Big[\big(1+(-1)^m\big)(m+1)^\ell - \big(1-(-1)^m\big)m^\ell\Big]\,p\big(n-t_s(m)\big) = 2\sigma^{(\ell)}_{s,o}\mathrm{mex}(n).
\]
Letting $\epsilon_m=\frac{1+(-1)^m}{2}$,  we get
\[\sigma^{(\ell)}_{s,o}\mathrm{mex}(n) = \sum_{m\ge 0}\Big[\epsilon_m(m+1)^\ell - (1-\epsilon_m)m^\ell\Big]\,p\big(n-t_s(m)\big).\]

Now, subtracting the two identities of Theorem~\ref{thm:moments}, and using
\[
\sigma^{(\ell)}_s\mathrm{mex}(n) - \overline{\sigma}^{(\ell)}_s\mathrm{mex}(n) = 2\sigma^{(\ell)}_{s,e}\mathrm{mex}(n),
\] we get
\[\sigma^{(\ell)}_{s,e}\mathrm{mex}(n) = \sum_{m\ge 0}\Big[(1-\epsilon_m)(m+1)^\ell - \epsilon_m m^\ell\Big]\,p\big(n-t_s(m)\big).\]
\end{proof}

\begin{proof}[Proof of Corollary~\ref{cor:oe}]
Putting $\ell=1$ in \eqref{eq:o-mom}, we have
\[
\sigma_{s,o}\mathrm{mex}(n)=\sigma^{(1)}_{s,o}\mathrm{mex}(n)
=\sum_{m=0}^{\infty}\Big\{\epsilon_m(m+1)-(1-\epsilon_m)m\Big\}\,p\big(n-t_s(m)\big).
\]
Since $\epsilon_m=1$ for $m$ even and $\epsilon_m=0$ for $m$ odd, we can write the previous equality as
\begin{align*}
	\sigma_{s,o}\mathrm{mex}(n)
	&=\sum_{\substack{m=0\\ m\ \mathrm{even}}}^{\infty}(m+1)\,p\big(n-t_s(m)\big)
	-\sum_{\substack{m=1\\ m\ \mathrm{odd}}}^{\infty}m\,p\big(n-t_s(m)\big)\\
	&=\sum_{m=0}^{\infty}(2m+1)\,p\big(n-t_s(2m)\big)
	-\sum_{m=0}^{\infty}(2m+1)\,p\big(n-t_s(2m+1)\big)\\
	&=\sum_{m=0}^{\infty}(2m+1)\Big\{p\big(n-t_s(2m)\big)-p\big(n-t_s(2m+1)\big)\Big\}.
\end{align*}
Similarly, putting $\ell=1$ in \eqref{eq:e-mom}, we have
\[
\sigma_{s,e}\mathrm{mex}(n)=\sigma^{(1)}_{s,e}\mathrm{mex}(n)
=\sum_{m=0}^{\infty}\Big\{(1-\epsilon_m)(m+1)-\epsilon_m m\Big\}\,p\big(n-t_s(m)\big).
\]
In view of the definition of $\epsilon_m$, we have
\begin{align*}
	\sigma_{s,e}\mathrm{mex}(n)
	&=\sum_{\substack{m=1\\ m\ \mathrm{odd}}}^{\infty}(m+1)\,p\big(n-t_s(m)\big)
	-\sum_{\substack{m=2\\ m\ \mathrm{even}}}^{\infty}m\,p\big(n-t_s(m)\big)\\
	&=\sum_{m=0}^{\infty}(2m+2)\,p\big(n-t_s(2m+1)\big)
	-\sum_{m=0}^{\infty}(2m+2)\,p\big(n-t_s(2m+2)\big)\\
	&=\sum_{m=0}^{\infty}(2m+2)\Big\{p\big(n-t_s(2m+1)\big)-p\big(n-t_s(2m+2)\big)\Big\}. \qedhere
\end{align*}
\end{proof}

\begin{proof}[Proof of Theorem~\ref{thm:main}]
By Proposition~\ref{lem:monotone}, both $\{\sromex(n)\}$ and $\{\sremex(n)\}$ are weakly increasing sequences of non-negative integers, and from equations \eqref{eq:gf-romex}, \eqref{eq:gf-remex} their generating functions have nonnegative coefficients and radius of convergence $1$. By Lemma~\ref{lem:gf-asym}, both satisfy the hypothesis of Theorem~\ref{thm:Ingham} with
\[
\mu=\frac{1}{4\sqrt{s}}, \qquad \nu=0, \qquad \lambda=\frac{\pi^2}{6}.
\]
Applying Theorem~\ref{thm:Ingham} to $\sromex(n)$ and $\sremex(n)$,
\[
\sromex(n)\sim\sremex(n)\sim\frac{\mu}{2\sqrt{\pi}}\,\lambda^{1/4}\,n^{-3/4}e^{2\sqrt{\lambda n}}.
\]
Here $2\sqrt{\lambda n}=\pi\sqrt{2n/3}$. For the constant,
\[
\frac{\mu}{2\sqrt{\pi}}\lambda^{1/4}
= \frac{1}{4\sqrt{s}}\cdot\frac{1}{2\sqrt{\pi}}\cdot\left(\frac{\pi^2}{6}\right)^{1/4}
= \frac{1}{8\sqrt{s}}\cdot\frac{\pi^{1/2}}{\sqrt{\pi}\,6^{1/4}}
= \frac{1}{8\sqrt{s}\,6^{1/4}}
= \frac{1}{8}\left(6s^2\right)^{-1/4}.
\]
Hence
\[
\sromex(n)\sim\sremex(n)\sim\frac{1}{8}\left(6n^3s^2\right)^{-1/4}\exp\left(\pi\sqrt{\frac{2n}{3}}\right),
\]
which is the stated formula.
\end{proof}

\begin{proof}[Proof of Theorem~\ref{thm:HSS}]
	We have,
	\[
	M(z,q)=\sum_{n=0}^{\infty}\sum_{m=1}^{\infty}p_{\rmex}(m,n)\,z^m q^n
	=\frac{1}{(q;q)_\infty}\sum_{m=1}^{\infty}z^m q^{t_s(m-1)}\big(1-q^{sm}\big).
	\]
	Taking $m=2k+1$, we obtain
\begin{equation}\label{eq1.8.0}
	\sum_{n=0}^{\infty}\sum_{k=0}^{\infty}p_{\rmex}(2k+1,n)\,z^{2k+1}q^n
	=\frac{1}{(q;q)_\infty}\sum_{k=0}^{\infty}z^{2k+1}q^{t_s(2k)}\big(1-q^{s(2k+1)}\big).
\end{equation}
Let us substitute $z=i$ in \eqref{eq1.8.0}
\begin{equation}\label{eq1.8.1}
	\sum_{n=0}^{\infty}\sum_{k=0}^{\infty}p_{\rmex}(2k+1,n)(-1)^k q^n
	=\frac{1}{(q;q)_\infty}\sum_{k=0}^{\infty}(-1)^k q^{t_s(2k)}\big(1-q^{s(2k+1)}\big).
\end{equation}
The left-hand side of \eqref{eq1.8.1} becomes
\begin{align}
	\sum_{n=0}^{\infty}\sum_{k=0}^{\infty}p_{\rmex}(2k+1,n)(-1)^k q^n
	&=\sum_{n=0}^{\infty}\left(\sum_{k=0}^{\infty}p_{\rmex}(4k+1,n)-\sum_{k=0}^{\infty}p_{\rmex}(4k+3,n)\right)q^n \notag\\
	&=\sum_{n=0}^{\infty}\Big(o^{(s)}_1(n)-o^{(s)}_3(n)\Big)q^n. \label{eq1.8.2}
\end{align}
Since $t_s(2k)+s(2k+1)=t_s(2k+1)$,  the right-hand side of \eqref{eq1.8.1} becomes
\begin{align} 
	\frac{1}{(q;q)_\infty}\sum_{k=0}^{\infty}(-1)^k q^{t_s(2k)}\big(1-q^{s(2k+1)}\big)
	&=\frac{1}{(q;q)_\infty}\left(\sum_{k=0}^{\infty}(-1)^k q^{t_s(2k)}-\sum_{k=0}^{\infty}(-1)^k q^{t_s(2k+1)}\right) \notag\\
	&=\frac{1}{(q;q)_\infty}\left(\sum_{j=0}^{\infty}(-1)^j q^{sj(2j+1)}+\sum_{j=-\infty}^{-1}(-1)^j q^{sj(2j+1)}\right) \notag\\
	&=\frac{1}{(q;q)_\infty}\sum_{j=-\infty}^{\infty}(-1)^j q^{sj(2j+1)}. \label{eq1.8.3}
\end{align}
%	where the rightmost sum in the second equality arises by substituting $j=-(\ell+1)$. Indeed, $t_s(2\ell)=s\ell(2\ell+1)$, and for $j=-(\ell+1)$ we have $sj(2j+1)=s(\ell+1)(2\ell+1)=t_s(2\ell+1)$ and $(-1)^j=-(-1)^\ell$.
	
 Now putting \eqref{eq1.8.2} and \eqref{eq1.8.3} into \eqref{eq1.8.1} gives us
	\begin{equation}\label{eq1.8.4}
		\sum_{n=0}^{\infty}\Big(o^{(s)}_1(n)-o^{(s)}_3(n)\Big)q^n
		=\frac{1}{(q;q)_\infty}\sum_{j=-\infty}^{\infty}(-1)^j q^{sj(2j+1)}.
	\end{equation}
	Replacing $a$ by $-q^{3s}$ and $b$ by $-q^{s}$ in Lemma~\ref{lem:JTP}, we have
	\[
	\sum_{j=-\infty}^{\infty}\big(-q^{3s}\big)^{\frac{j(j+1)}{2}}\big(-q^{s}\big)^{\frac{j(j-1)}{2}}
	=(q^{3s};q^{4s})_\infty(q^{s};q^{4s})_\infty(q^{4s};q^{4s})_\infty.
	\]
	That means
	\begin{equation}\label{eq1.8.5}
		\sum_{j=-\infty}^{\infty}(-1)^j q^{2sj^2+sj}
		=\sum_{j=-\infty}^{\infty}(-1)^j q^{sj(2j+1)}
		=(q^{s};q^{4s})_\infty(q^{3s};q^{4s})_\infty(q^{4s};q^{4s})_\infty.
	\end{equation}
	Now from \eqref{eq1.8.5} and \eqref{eq1.8.4}, we get
	\begin{align*}
		\sum_{n=0}^{\infty}\Big(o^{(s)}_1(n)-o^{(s)}_3(n)\Big)q^n
		&=\frac{1}{(q;q)_\infty}\times(q^{s};q^{4s})_\infty(q^{3s};q^{4s})_\infty(q^{4s};q^{4s})_\infty\\
		&=\frac{(q^{s};q^{2s})_\infty\,(q^{4s};q^{4s})_\infty}{(q;q)_\infty}
		=\frac{f_s f_{4s}}{f_1 f_{2s}}.
	\end{align*}
	This completes the proof.
\end{proof}

\subsection*{Data availability statement} We have used no data in our manuscript.

%====================================================================

\end{document}